\documentclass[10pt, twoside]{amsart}
\usepackage[all]{xy}   
        \usepackage {amssymb,latexsym,amsthm,amsmath}
        \usepackage{enumitem}
        
\long\def\symbolfootnote[#1]#2{\begingroup%
\def\thefootnote{\fnsymbol{footnote}}\footnote[#1]{#2}\endgroup}

\makeatletter
\def\imod#1{\allowbreak\mkern10mu({\operator@font mod}\,\,#1)}
\makeatother

\newtheorem{theorem}{Theorem}[section]
\newtheorem{lemma}[theorem]{Lemma}
\newtheorem{corollary}[theorem]{Corollary}
\newtheorem{proposition}[theorem]{Proposition}

\theoremstyle{definition}

\newtheorem{example}[theorem]{Example}
\numberwithin{equation}{section}

\newcommand{\ignore}[1]{}

\newcommand{\mynote}[1]{}

\newtheorem*{theoremA}{Theorem~A}
\newtheorem*{theoremB}{Theorem~B}
\newtheorem*{corollaryC}{Corollary~C}

\begin{document}
\setcounter{section}{0}
\title{$p$-Power conjugacy classes in $U(n,q)$ and $T(n,q)$}
\author{Silvio Dolfi}
\address{Dipartimento di Matematica e Informatica Dini, Universit\`a di Firenze, 50134 Firenze, Italy}
\email{dolfi@math.unifi.it}
\author{Anupam Singh}
\address{IISER Pune, Dr. Homi Bhabha Road, Pashan, Pune 411008 India}
\email{anupamk18@gmail.com}
\author{Manoj K. Yadav}
\address{School of Mathematics, Harish-Chandra Research Institute, HBNI, Chhatnag Road, Jhunsi, Allahabad 211019, India}
\email{myadav@hri.res.in}
\thanks{The first named author gratefully acknowledges the hospitality of Harish-Chandra Research Institute (Allahabad) and IISER Pune. The second named author would like to acknowledge support of SERB-MATRICS grant and the hospitality of the Department of Mathematics, University of Florence, Italy during his visits.}
\subjclass[2010]{20G40}
\today
\keywords{Word map, triangular group, unitriangular group}


\begin{abstract}
Let $q$ be a $p$-power where $p$ is a fixed prime. In this paper, we look at the $p$-power maps on unitriangular group $U(n,q)$ and triangular group $T(n,q)$. In the spirit of Borel dominance theorem for algebraic groups, we show that the image of this map contains large size conjugacy classes. For the triangular group we give a recursive formula to count the image size.  
\end{abstract}
\maketitle
\section{Introduction}
Let $G$ be a finite  group and $w$ be a word. The word $w$ defines a map into  $G$ called a word map. It has been a subject of intensive investigation whether these maps are surjective on finite simple and quasi-simple groups;  we refer to the article by Shalev~\cite{sh} for a survey on this subject. 
A more general problem is to determine  the image $w(G)$ of a word map and, in particular,  its size.
In this paper we investigate power maps, that is, maps corresponding to the word  $w = X^p$, for the  lower-triangular matrix group $T(n,q)$ and lower unitriangular matrix group $U(n,q)$ over finite fields $\mathbb F_q$, where $q$ is a $p$-power for a fixed prime $p$. Results concerning the verbal subgroup, that is the group generated by the image of the power map, for triangular and unitriangular group can be found in~\cite{bi} and~\cite{so}.

Motivated by the Borel's dominance theorem for algebraic groups, Gordeev, Kunyavski\u{i} and Plotkin started investigating the image of a non-surjective word map more closely (see~\cite{gkp, gkp1, gkp2, gkp3}). In the spirit of questions raised in  \cite[Section 4]{gkp3} for algebraic groups,  we address, for the groups $T(n,q)$ and $U(n,q)$, the question: Which semisimple, and unipotent elements lie in the image of the power maps and whether it contains  `large'   conjugacy classes?

One of the motivations for our interest in the triangular and unitriangular groups lies in the fact that  $T(n,q)$ is a Borel subgroup of $GL(n,q)$ and $U(n,q)$ is a Sylow $p$-subgroup of $GL(n,q)$. In the finite groups of Lie type, the regular semisimple elements play an important role as they are dense (see~\cite{gksv}). Considering the image of a word map on maximal tori has turned out to be useful in getting asymptotic results. Thus, we aim at considering the large size conjugacy classes in $U(n,q)$, described in \cite{va2}, and try to understand if they are in the image under the power map $w=X^p$. (Note that, clearly, raising to a power coprime to $p$ gives a bijection of $U(n,q)$).
In what  follows, we  use the notation $G^p$ for the image $w(G)$ of a group $G$ under the word map given by $w = X^p$ (we call it \emph{power map}). So, $G^p = \{ g^p \mid g \in G \}$ is the  \emph{set} consisting of the $p$-powers of the elements of $G$. We remark that the \emph{verbal width} with respect to power maps, that is, the smallest number $k$ such that the product of $k$-copies of $G^p$, coincides with the verbal subgroup $\langle G^p \rangle$,  has already been determined: see~\cite[Theorem~5]{bi} for $G = U(n,q)$ and~\cite[Theorem~1]{so} for $G = T(n,q)$.

It is known (see~\cite[Theorem 3]{bi} or Proposition~\ref{BU})  that $U(n,q)^p$ is contained in the subgroup
$U_{p-1}(n,q)=\left \{(a_{ij})\in U(n,q) \mid a_{ij}=0,\ \forall\ i-j\leq p-1 \right\}$
consisting of the lower triangular matrices with the first $p-1$ sub-diagonals having zero entries.
Moreover,  $U(n,q)^p = 1$ if and only if $n\leq p$, and  $U(n,q)^ p = U_{p-1}(n,q)$ if and only if $n=p+1$ and $p+2$. Our first result, for $n \ge p+3$, is the following estimate on the set of pth powers in $U(n,q)$. 
\begin{theoremA}
Let $q$ be a power for a prime $p$ and $n$ an integer such that  $n\geq p+3$. Then, the set $U(n,q)^p$ is a proper generating subset of $U_{p-1}(n,q)$
and $|U(n, q)^p| > \frac{1}{3}|U_{p-1}(n,q)|$ when $q \geq n-p-1$.
\end{theoremA}

Next we prove the following result, which reduces the counting of $p$-powers for $T(n,q)$ to that of unitriangular groups of smaller size.
\begin{theoremB}
Let $q$ be a $p$-power and suppose $q>2$. Then for the group $T=T(n,q)$ we have,
$$|T^p|=\sum_{(a_1,\ldots, a_k)\vdash n, k<q} \left( \frac{(q-1)\cdots (q-k) n!}{ \prod_{b=1}^n m_b!(b!)^{m_b}}\right)\left( \prod_{i=1}^k |U(a_i,q)^p|\right) q^{{n\choose 2}-\sum_{i=1}^k {a_i\choose 2}} $$
where the $m_b$'s are obtained by writing the partition $(a_1, \ldots, a_k)$ in power notation as $1^{m_1}\ldots n^{m_n}$. 
\end{theoremB}

Using the estimate in Theorem~A, we hence get 
\begin{corollaryC} 
Let  $q$ be a power of a prime $p$ such that $q>n-p-1$. Then for the group $T=T(n,q)$ we have,
$$\frac{|T^p|}{|T|}\geq \frac{2^{n-2}}{9(q-1)^{n-2} q^{ (p-1)(n-p)}}.$$
\end{corollaryC}

\vspace{.2in}

We conclude the section with a quick layout. Theorem~A  is proved in Section~3 and Theorem~B and Corollary~C in Section~4. All groups considered in what follows are tacitly assumed to be finite.

\section{Conjugacy classes in $U(n,q)$}

The conjugacy classes of  the unitriangular group $U(n,q)$, considered as the group of upper unitriangular matrices,  have been studied in a series of papers by Arregi and Vera-L\'{o}pez; we will use, in particular, the results in \cite{va1,va2}. For the convenience of the reader, we reproduce some notations and  results from  \cite{va1} in the setting of lower unitriangular matrices, i.e., swapping the notation by taking transpose. 

 Let us order the index set $\mathcal{I} = \{(i, j) \mid 1 \le j \le i \le n\}$ in the following manner:
  $$(n, n-1) < (n-1, n-2) < (n, n-2) < (n-2, n-3) < \cdots < (n-1, 1) < (n, 1).$$
  To every $A = (a_{ij}) \in U(n,q)$ and $(r,s) \in \mathcal{I}$, one  associates a vector $\mu_{(r,s)}(A)$ (the $(r, s)$-\emph{weight} of $A$) as follows:
  $$\mu_{(r, s)}(A) := \big( \mu(a_{ij}) \big)_{(i, j) \le (r, s)},$$ 
  where $\mu(a_{ij}) = 0$ if $a_{ij} = 0$ and $\mu(a_{ij}) = 1$ if $a_{ij} \ne 0$. $\mu_{(n, 1)}(A)$ is called the \emph{weight} of $A$ and is simply denoted by $\mu(A)$.
  So, $\{\mu(A) \mid A \in U(n,q) \}= \{0,1\}^{\frac{n(n-1)}{2}}$ and  we totally order this set  of weights by  lexicographical order (considering $0 < 1$). For a given index $(r, s) \in \mathcal{I}$, we  order $\mu_{(r, s)}(A)$ in the same manner. We remark that in \cite{va1}, the word `type' is used in place of `weight'. But we will use `weight' as we use `type' for some other purpose. 
  
  For $(r, s) \in \mathcal{I}$, define 
  $$\mathcal{G} _{(r, s)} := \{A = (a_{ij}) \in U(n,q) \mid a_{ij} = 0 ~\mbox{for all}~ (i, j) \le (r, s)\}.$$
  It is a routine check to see that  $\mathcal{G} _{(r, s)}$ is a normal subgroup of $U(n,q)$ having order $q^{ns-r -\frac{s(s-1)}{2}}$.
  As proved in \cite[Theorem 3.2]{va1},  every conjugacy class in $U(n,q)/\mathcal{G} _{(r, s)}$ contains a unique element of minimum  $(r, s)$-weight.
  A matrix $A \in U(n,q)$ is said to be \emph{canonical} if $A\mathcal{G} _{(r, s)}$ is the unique element of its conjugacy class in 
  $U(n, q)/ \mathcal{G} _{(r, s)}$ having minimal $(r, s)$-weight for all $(r, s) \in \mathcal{I}$.

  For each $(r,s) \in \mathcal{I}$, let us define 
  $$\mathcal{N}_{(r,s)} := \mathcal{G}_{(r, s)^*}/ \mathcal{G}_{(r, s)},$$
  where $(r, s)^*$ denotes the preceding pair of $(r, s)$ in the ordering of $\mathcal{I}$ defined above. It follows from \cite[Lemma 3.4]{va1} that for every $A \in U(n, q)$  and $(r, s) \in \mathcal{I}$ the number of conjugacy classes in  $U(n, q)/ \mathcal{G} _{(r, s)}$ which intersect with $\bar{A} \mathcal{N}_{(r,s)}$ is either $1$ or $q$, where $\bar{A} = A\mathcal{G}_{(r, s)}$. We say that $(r, s) \in \mathcal{I}$ is an \emph{inert point} of $A \in U(n, q)$ if the number in the preceding statement is $1$.

The following two results are restatements of \cite[Lemma 3.7, Lemma 3.8]{va1} for lower unitriangular matrices.

\begin{lemma}\label{dual1}
Let $A \in U(n, q)$ be a canonical matrix such that $a_{rs} \neq 0$ and $a_{js} = 0$ for all $j$ such that $s < j < r$. Then  the pairs $(r, s')$, with $s' < s$, are inert points of $A$.
\end{lemma}
  
\begin{lemma}\label{dual2}
Let $A \in U(n, q)$ be a canonical matrix such that $a_{rs} \neq 0$ and $a_{ri} = 0$ for all $i$, $s < i < r$. Then  the pair $(r', s)$ for any  $r' > r$ is an inert point of $A$ if $a_{jr'} = 0$ for all $j > r'$.
\end{lemma}

  We set the following notation. Given  $k  \in \{0, 1, \ldots, n-1\}$, 
we say that the array of  entries $(a_{k+1, 1}, a_{k+2, 2}, \ldots, a_{n, n-k})$ is the \emph{$k^{th}$-sub-diagonal} of the matrix $A = (a_{i, j})$. 
  For $l$ such that $0\leq l\leq n-1$, define
$$U_l(n,q):=\left \{A=(a_{ij})\in U(n,q) \mid a_{ij}=0,~ \mbox{for all} ~  i-j \leq l \right\}$$
consisting of lower unitriangular matrices whose first $l$ sub-diagonals have all zero entries. We remark that  
$$U(n,q) = U_0(n,q)  \supset U_1(n,q) \supset \cdots \supset U_l(n,q) \supset \cdots \supset U_{n-1}(n,q)=\{1\}$$
is the lower central  series of $U(n,q)$, with $U_l(n,q) = \gamma_{l+1}\big(U(n,q)\big)$, and that the $U_l(n,q)$  are  the only fully invariant  subgroups
of $U(n,q)$ (\cite[Theorem 1]{bi}).

Having fixed a dimension $n$, in $GL(n,q)$ we denote by $I$ the identity matrix and by $e_{rs}$ the elementary matrix with $1$ at $(r,s)^{th}$ place and $0$ elsewhere.
We now turn our attention to some relevant elements of the subgroups $U_l(n,q)$. 

For  $0 \leq l \leq n-2$, set
\begin{equation}\label{canonical-elements}
  A(a_1, a_2, \ldots, a_{n-l-1}) =  I+ \displaystyle\sum_{i=1}^{n-l-1} a_i e_{l+1+i,i} \in U_l(n,q),
  \end{equation}
  where $a_1, a_2, \ldots , a_{n-l-1} \in \mathbb F_q$.
  
We have the following important property of the elements defined in~(\ref{canonical-elements})
\begin{lemma}\label{UL-Canonical}
For every choice of $0 \leq l \leq n-2$ and $a_1, a_2, \ldots , a_{n-l-1} \in \mathbb F_q$,  the element $ A(a_1, a_2, \ldots, a_{n-l-1}) $ is a canonical element of $U(n,q)$.
\end{lemma}

\begin{proof}
  In order to show that  $A = A(a_1, a_2, \ldots, a_{n-l-1}) \in U_l(n,q)$ is a canonical element of $U(n,q)$, we need to prove 
  that each  non-zero entry on the $l+1$-th subdiagonal of $A\mathcal{G}_{(r,s)}$ will continue to be  non-zero in every $U(n,q)/\mathcal{G}_{(k,l)}$-conjugate  of $A\mathcal{G}_{(r,s)}$ for all $(r, s) \in \mathcal{I}$.
  More generally, we observe that  if $A = (a_{i, j}) \in U_l(n,q)$ and $B = (b_{i, j})  \in U(n,q)$, then the $(l+1)$-th subdiagonal of $A$ and $B^{-1}AB$ are identical modulo $\mathcal{G}_{(r,s)}$ for all pairs $(r, s) \in \mathcal{I}$.
  In fact, it is readily checked that the element in the  $(l+1+k, k)$-th place, for $k = 1, \ldots, n-l-1$, of the $(l+1)$-th subdiagonal of both $AB$ and $BA$,  is  simply $a_{l+1+k, k} + b_{l+1+k, k}$ modulo $\mathcal{G}_{(r,s)}$. This shows that $A$ is canonical in $U(n,q)$.
\end{proof}

We conclude this section with the following result in which we single out conjugacy classes of $U_l(n,q)$ of considerably large orders, including the largest ones.
\begin{proposition}\label{U-classes}
  Let $0 \leq l \leq n-2$. For $0 \le m \le \lfloor{\frac{n-l-1}{2}}\rfloor +1$, set   
  $$\mathcal{A}_m = \{ A(a_1, a_2, \ldots, a_{n-l-1}) \mid a_i = 0 ~\mbox{for}~ i \le m, a_{i}  \in F_q^{\times} ~\mbox{for}~ i > m\}$$ 
  and for $\lfloor{\frac{n-l-1} {2}}\rfloor < m \le n-l$, set
  $$\mathcal{B}_m = \{A(a_1, a_2, \ldots, a_{n-l-1}) \mid a_i \in F_q^{\times} ~\mbox{for}~ i < m, a_{i} =0  ~\mbox{for}~ i \ge m\}.$$  
  Then, the elements in $\mathcal{A}_m$ are representatives of distinct $U(n,q)$-conjugacy classes of size $q^{\frac{(n-l-1)(n-l-2)}{2} - \frac{m(m-1)}{2}}$ and the elements in $\mathcal{B}_m$ are representatives of distinct $U(n,q)$-conjugacy classes of size $q^{\frac{(n-l-1)(n-l-2)}{2} - \frac{(n-l-m)(n-l-m-1)}{2}}$.
\end{proposition}
\begin{proof}
Since, by Lemma~\ref{UL-Canonical}, the elements in $\mathcal{A}_m$  and $\mathcal{B}_m$ are canonical elements of $U(n, q)$, it follows by~\cite[Corollary 3.3]{va1}  that these are pair-wise non-conjugate in $U(n,q)$.

 Let $A \in \mathcal{A}_m$. Then for each $t$, $m+1 \le t \le n-l-1$, it follows from Lemma \ref{dual1} that there are  $t-1$ inert points of $A$ corresponding to $a_t$. So the  number of inert points of $A$ is at least $\frac{(n-l-1)(n-l-2)}{2} - \frac{m(m-1)}{2}$.  Thus, by \cite[Theorem 3.5]{va1}, the conjugacy class of $A$ in $U(n, q)$ has size at least $q^{\frac{(n-l-1)(n-l-2)}{2} - \frac{m(m-1)}{2}}$. We claim that it can not be bigger than this. Let $G_m$ denote the subset of 
 $U_{l+1}(n,q)$ defined as 
 $$G_m = \{B = (b_{ij}) \in U_{l+1}(n,q) \mid b_{ij} = 0 ~ \mbox{for all} ~ l+2 < i \le l+m+1, 1 \le j < m\}.$$
 It is not difficult to see that $G_m$ is a normal subgroup of $U(n, q)$ having order \\$q^{\frac{(n-l-1)(n-l-2)}{2} - \frac{m(m-1)}{2}}$. Notice that $[A, U(n, q)] \subseteq G_m$, where 
 $$[A, U(n, q)] = \{[A, C] \mid C \in U(n, q)\}.$$
 This shows that the size of the conjugacy class of $A$ in $U(n, q)$ is at the most $|G_m|$, as claimed. Hence the assertion for the elements of  $\mathcal{A}_m$ holds.
 
 Assertion for the elements in $\mathcal{B}_m$ holds on the same lines using Lemma \ref{dual2}, which completes  the proof.
\end{proof}

\section{Unitriangular matrix group}

We look at the power map $w=X^p$ on the unitriangular group $U(n,q)$. We begin by stating  the following  results from~\cite{bi} to improve readability of this section.
\begin{lemma}\label{m-power}
Let $A$ be a lower unitriangular matrix in $U(n, q)$ such that $A - I  = (a_{ij})$. Then the matrix $A^m= I + (b_{ij})$ is given by:
\begin{eqnarray*}\label{power}
b_{ij}&=& {m\choose 1} a_{ij} + {m\choose 2}\left(\sum_{r_1=j+1}^{i-1} a_{i,r_1}a_{r_1,j}\right) + {m\choose 3} \left(\sum_{r_1=j+1}^{i-1} \sum_{r_2=r_1+1}^{i-1} a_{i,r_2}a_{r_2,r_1}a_{r_1,j}\right)\\
& & +\cdots   \cdots + {m\choose k}\left(\sum_{r_1=j+1}^{i-1} \sum_{r_2=r_1+1}^{i-1}\cdots\sum_{r_{k-1}=r_{k-2}+1}^{i-1} a_{i,r_{k-1}}a_{r_{k-1},r_{k-2}}\cdots a_{r_1,j} \right)\\
&& + \cdots 
 \cdots + {m\choose m}\left( \sum_{r_1=j+1}^{i-1} \sum_{r_2=r_1+1}^{i-1}\cdots\sum_{r_{m-1}=r_{m-2}+1}^{i-1} a_{i,r_{m-1}}a_{r_{m-1},r_{m-2}}\cdots a_{r_1,j}\right). 
\end{eqnarray*}
\end{lemma}

We use Lemma~\ref{m-power} to prove the following result for $p$th powers.

\begin{corollary}\label{p-power-utri}
Let $A   \in U(n,q)$ be such that  $A - I = (a_{i, j})$ and $A^p = I + (b_{i, j})$. Then,  $b_{i,j}=0$ for all $i-j<p$ and 
$$  b_{i,j}=  \sum_{r_1=j+1}^{i-1} \sum_{r_2=r_1+1}^{i-1}\cdots\sum_{r_{p-1}=r_{p-2}+1}^{i-1} a_{i,r_{p-1}}a_{r_{p-1},r_{p-2}}\cdots a_{r_1,j}$$
otherwise. In particular, if $n\leq p$, then $A^p=I$, and if $n>p$, then  $U(n, q)^p \subseteq U_{p-1}(n,q)$.
\end{corollary}
\begin{proof}
Since the binomial coefficients appearing in the formula of Lemma~\ref{m-power} for $m = p$ are all zero modulo $p$, except  possibly the last one,  we get 
$$b_{i,j}=  \sum_{r_1=j+1}^{i-1} \sum_{r_2=r_1+1}^{i-1}\cdots\sum_{r_{p-1}= r_{p-2}+1}^{i-1} a_{i,r_{p-1}}a_{r_{p-1},r_{p-2}}\cdots a_{r_1,j} .$$ 
If $i-j < p$, this is an empty sum, that is, it's $0$. This  happens for all pairs  $(i,j)$ if $n < p$; giving $A^p = I$. If $i-j \geq p$, which actually  implies that $n \ge p$, then $a_{i,j}$'s are given by the expression as stated, and obviously fall in $U_{p-1}(n,q)$. 
\end{proof}

As an immediate consequence, we have the following result.

\begin{proposition}\label{pthpower}
For $n > p$ and $l = p-1$, every element of $\mathcal{A}_m$ and  $\mathcal{B}_m$ (defined in Proposition \ref{U-classes}) is a pth power in $U(n, q)$.
\end{proposition}
\begin{proof}
We first show that the elements  $A:=A(a_1, a_2, \ldots, a_{n-l-1})$ defined in \eqref{canonical-elements} for $a_1, a_2, \ldots,$ $ a_{n-l-1} \in \mathbb F_q^{\times}$
are $p$th powers. Let $b_{2,1} = \cdots = b_{p, p-1} = 1$. Then iteratively define 
$$b_{p+i+1, p+i} := (b_{i+2, i+1}  \cdots  b_{p+i, p+i-1})^{-1}a_{i+1},$$
for $0 \le i < n-p$.  Now consider the lower unitriangular matrix  $C := (c_{i,j})$, where  $c_{i, i-1} = b_{i, i-1}$ for $2 \le i \le n$ and $c_{i, j} = 0$ for $i - j > 1$. Using Corollary \ref{p-power-utri}, it is a routine computation to show that $C^p = A$. 

Now let $A :=A(a_1, a_2, \ldots, a_{n-l-1}) \in \mathcal{B}_m$. Then, by the definition, $a_i \in F_q^{\times} ~\mbox{for}~ i < m, a_{i} =0  ~\mbox{for}~ i \ge m$. Thus, in the above procedure, $b_{i, i-1} = 0$ for $p + m \le i \le n$. Considering  $C := (c_{i,j})$, where  $c_{i, i-1} = b_{i, i-1}$ for $2 \le i \le n$ and $c_{i, j} = 0$ for $i - j > 1$, we see, again using  Corollary \ref{p-power-utri}, that $C^p = A$. 

For $A  :=A(a_1, a_2, \ldots, a_{n-l-1}) \in \mathcal{A}_m$, let $b_{i, i-1} = 0$ for $2 \le i \le m+1$, $b_{i, i-1} = 1$ for $m+2 \le i \le m+p$ and then iteratively define
$$b_{p+i+1, p+i} := (b_{i+2, i+1}  \cdots  b_{p+i, p+i-1})^{-1}a_{i+1},$$
for $m \le i < n-p$. Again considering $C := (c_{i,j})$, where  $c_{i, i-1} = b_{i, i-1}$ for $2 \le i \le n$ and $c_{i, j} = 0$ for $i - j > 1$, it follows that  $C^p = A$, which completes the proof.
\end{proof}

The following proposition, which follows from the above formulas, is proved in~\cite[Theorem~2, Theorem~3]{bi} (also see \cite[III, Satz 16.5]{Hu}). 
\begin{proposition}\label{BU}
Let $q$ be a $p$-power.  Then,
\begin{enumerate}
\item for $n\leq p$,  $U(n,q)^p = 1$;
\item for $n=p+1$ and $n = p+2$,  $U(n,q)^p = U_{p-1}(n,q)$; 
\item for $n \geq  p+3$, $U(n,q)^p \subset U_{p-1}(n,q)$ and   $\langle U(n,q)^p \rangle = U_{p-1}(n,q)$.
\end{enumerate}
\end{proposition}

We now provide a lower bound on $|U(n,q)^p|$.
\begin{proposition}\label{lbound}
Let $q$ be a  power of $p$ and $n$ an integer  such that $n \ge p+3$. Then,  if $n-p$ is even, 
\begin{eqnarray*}
|U(n,q)^p| \ge  q^{\frac{(n-p)(n-p-1)}{2}} \big( (q-1)^{n-p} \big)  + \displaystyle\sum_{m=1}^{ \lfloor{\frac{n-p}{2}}\rfloor} q^{\frac{(n-p)(n-p-1)}{2} - \frac{m(m-1)}{2}} \big( 2 (q-1)^{n-p-m}\big)
\end{eqnarray*}
and if $n-p$ is odd,
\begin{eqnarray*}
|U(n,q)^p| &\ge&  q^{\frac{(n-p)(n-p-1)}{2}} \big( (q-1)^{n-p} \big)  + \displaystyle\sum_{m=1}^{ \lfloor{\frac{n-p}{2}}\rfloor} q^{\frac{(n-p)(n-p-1)}{2} - \frac{m(m-1)}{2}} \big( 2 (q-1)^{n-p-m}\big)\\
& & + q^{\frac{(n-p)(n-p-1)}{2} - \frac{r(r-1)}{2}} \big(2 (q-1)^{n-p-r}\big),
\end{eqnarray*}
where $r =  \lfloor{\frac{n-p}{2}}\rfloor +1$.
\end{proposition}
\begin{proof}
It follows from Proposition \ref{pthpower} that every element of $\mathcal{A}_m$ as well as  of $\mathcal{B}_m$ is a $p$th power in $U(n, q)$. The result now follows by considering the sizes of all distinct conjugacy classes of elements of  $\mathcal{A}_m$ and  $\mathcal{B}_m$ obtained in Proposition \ref{U-classes}.
\end{proof}

 We are now ready  to prove Theorem~A.  

\begin{proof}[Proof of Theorem~A]
  The first assertion follows from Proposition \ref{BU}. For the second assertion, by Proposition \ref{lbound}, we have
$$|U(n,q)^p| > q^{\frac{(n-p)(n-p-1)}{2}} \left( (q-1)^{n-p} + 2 (q-1)^{n-p-1} \right).$$
Hence, 
$$\frac{|U(n,q)^p|}{|U_{p-1}(n,q)|} > \frac{q^{\frac{(n-p)(n-p-1)}{2}} \left( (q-1)^{n-p} + 2 (q-1)^{n-p-1} \right)}{q^{\frac{(n-p)(n-p+1)}{2}}},$$
which implies 
$$\frac{|U(n,q)^p|}{|U_{p-1}(n,q)|} > \left(1-\frac{1}{q}\right)^{n-p-1} \left(1+\frac{1}{q}\right).$$ 
Thus, if we take $q \geq n-p-1$, then  we get 
$$\frac{|U(n,q)^p|}{|U_{p-1}(n,q)|} > \left(1-\frac{1}{q}\right)^{n-p-1}\left(1+\frac{1}{q}\right) \ge \left(1-\frac{1}{q}\right)^{q}\left(1+\frac{1}{q}\right) > \frac{1}{3}.$$
This completes the proof.
\end{proof}
\noindent We conclude this section with some computations using MAGMA \cite{magma}, which  are as follows.
\vskip2mm 
\begin{center}
\begin{tabular}{|c|c|c|c|}\hline 
\ $(n,q)$ \ & \ $|U(n, q)^p|$\  & $\ |U_{p-1}(n,q)|$\  &\ $|U(n,q)^p|/|U_{p-1}(n,q)|\ $\\ \hline
(5,2) & 52 & $2^6=64$ & $> \frac{1}{3}$ \\\hline
(5,4) & 3376 & $4^6=4096$ & $> \frac{1}{3}$ \\\hline
(6,2) & 600 & $2^{10}=1024$ & $> \frac{1}{3}$ \\\hline
(6,3) & 585 & $3^6=729$ & $> \frac{1}{3}$ \\\hline
(7,2) & 13344 & $2^{15}=32768$ & $> \frac{1}{3}$\\\hline
(8,2) & 573184 & $2^{21} = 2097152$ & $< \frac{1}{3} $\\\hline 
\end{tabular}
\end{center}
\vskip2mm
\noindent In view of the values in the last row of this table, we remark that the condition  on $q$ in Theorem A can not be completely dropped.    

\section{Triangular matrix group}
In this section we consider  the group of triangular matrices  $T(n,q)$, where $q$ is a power of a prime $p$, 
aiming at computing the size $|T(n, q)^p|$ of the set of its $p$-powers.
Since the group $T(n,2) = U(n,2)$ we assume $q> 2$, now onwards.
We begin with setting up some notation.
We denote by $D(n, q)$ the subgroup of $T(n, q)$ consisting of the diagonal matrices. 
The elements of $D(n,q)$ can be grouped in ``types'' in such a way that all elements of each type have the isomorphic centralizers in $T(n,q)$.
We recall that a  \emph{partition}  of a positive integer $n$ is a sequence of positive integers $\delta = (a_1, \ldots, a_k)$ such that $a_1\geq a_2 \geq \ldots \geq a_k >0$ and
$\displaystyle{\sum_{j=1}^k} a_j=n$.
One can also write the partition $\delta$ in \emph{power notation} $1^{m_1}2^{m_2}\ldots n^{m_n}$ where $m_i = |\{a_j \mid a_j = i\}|$ is the number of
\emph{parts} $a_j$'s equal to $i$, for $1\leq i \leq n$; so, $m_i\geq 0$ and $\displaystyle{\sum_{i=1}^n} im_i =n$.

Let  $\Pi = \{X_1, X_2, \ldots , X_k \}$ be a \emph{set-partition} of  $I_n = \{1,2 , \ldots, n\}$, i.e.,  a family of non-empty and pairwise disjoint subsets of  $I_n$, whose union is $I_n$. Setting $a_i = |X_i|$ and assuming, as we may, that $a_i \geq a_j$ for $i \leq j$,  the  tuple
  $\delta = (a_1, a_2, \ldots, a_k)$ is a partition of the number $n$; we say that $\delta$ is the \emph{type}  of $\Pi$. 

A diagonal matrix $d \in D(n, q)$,  seen  as a map from the set $I_n = \{1,2 , \ldots, n\}$ to the set $F_q^{\times}$ of the non-zero elements of the field $F_q$,  
determines in a natural way a set-partition $\Pi_d$ of $I_n$, namely the family  of the non-empty fibers of the map $d$.
We set $\delta = \tau(\Pi_d)$ as the \emph{type} of $d$ and we write $\delta = \tau(d)$. 

We denote the number $k$ of parts in $\delta$ by  $l(\delta)$, the \emph{length} of $\delta$,  and we observe  
that there exist elements of type $\delta$ in $D$ if and only if  $l(\delta)  < q$.
(Thus not all partitions of $n$  may appear as type of an element in $D(n,q)$, when $q \leq n$).

Given a partition $\delta$ of $n$ with $l(\delta) <q$, we denote by  $D_{\delta}(n,q) = \{ d \in D \mid \tau(d) = \delta\}$  the set of all elements in $D(n,q)$ of the  given type $\delta$.
\begin{lemma}\label{type-count}
  Let $n$ be a positive integer and let  $\delta = (a_1, a_2, \ldots, a_k)$ be a partition of $n$. Write $\delta$ in power notation as $1^{m_1}\ldots n^{m_n}$ and assume
  that $k = l(\delta) < q$.  Then 
 $$ |D_{\delta}(n,q)| = \frac{(q-1)! n!}{(q-k-1)! \prod_{i=1}^n \left( (i!)^{m_i}\cdot m_i! \right)}   \; .$$ 
\end{lemma}
\begin{proof}
  Let $\Delta$ be the set consisting of the set-partitions of $I_n$ having type $\delta$. As above, we associate to  a diagonal element $d \in D_{\delta}$ a  set-partition
  $\Pi_d \in \Delta$, and observe that all the fibers of the map $\pi:D_{\delta} \rightarrow \Delta$ defined by $\pi(d) = \Pi_d$,  have the same size $\frac{(q-1)!}{(q-k-1)!}$ (the
  number of injective maps from a set of $k = l(\delta)$ elements into a set of $q-1$-elements).
  On the other hand, the cardinality $|\Delta|$ is easily determined by looking at the natural transitive action of the symmetric group $S_n$ on  $\Delta$ and observing that
  the stabilizer in $S_n$ of a partition of type $\delta = 1^{m_1}\ldots n^{m_n}$ has size $\prod_{i=1}^n \left( (i!)^{m_i}\cdot m_i! \right)$.
\end{proof}

\begin{lemma}\label{cent-sizes}
  For any partition $\delta$ of $n$, with $l(\delta) < q$, and for any element $d \in D_{\delta}(n,q)$, all  centralizers $C_{U(n,q)}(d)$ belong to the same isomorphism class.
Moreover, if  $\delta = (a_1, a_2, \ldots, a_k)$, then
 $\mid C_{U(n,q)}(d) \mid = q^{\sum_{i=1}^k {\binom{a_i}{2}}}.$  
\end{lemma}

\begin{proof}
  Let   $\delta = (a_1, a_2, \ldots, a_k)$ be a given partition of $n$, with $k < q$, and let
  $$d=(\underbrace{d_1, \ldots, d_1}_{a_1}, \ldots, \underbrace{d_k, \ldots, d_k}_{a_k}), $$
  where $d_1, \ldots, d_k$ are distinct elements of $F_q$, be a 'standard-form' element in $D_{\delta}$.
  Write $G = GL_n(q)$, $U = U(n,q)$. It is well known that  $ C_G(d) = \prod_{i= 1}^k GL_{a_i}(q)$,  the subgroup of $\delta$-block matrices.

  Let $b = b_{\pi} \in G$ be a permutation matrix, where $\pi \in S_n$.
  We will show that
  $C_U(d^b)$ is isomorphic to $C_U(d)$. In order to do this, it is enough to show that the two subgroups have the same order, since  $C_U(d) = C_G(d) \cap U$ is
  a Sylow $p$-subgroup of $C_G(d)$ and $C_U(d^b) = C_G(d)^b \cap U \cong C_G(d) \cap U^{b^{-1}}$ is a $p$-subgroup.  
  We denote by $M =  \prod_{i= 1}^k M_{a_i}(q)$,  the  $\delta$-blocks matrix algebra and write $C = C_G(d)$.  We observe that $M \cap U = C \cap U$,
  $M^b \cap U = C^b \cap U$ and that, arguing by induction on the number $k$ of diagonal blocks, in order to prove that $|M \cap U| = |M^b \cap U|$ we can
  reduce to the case $k = 2$.

  Now,  $M \cap U = (I_n + V) \cap U  = I_n + (V \cap U_0)$, where $V$ is  the $F_q$-space spanned by the set of pairs of  elementary matrices
  $\{e_{i,j}, e_{j,i}\}$, where $1 \leq i < j \leq a_1$ or $a_1+1 \leq i < j \leq n$, and the $F_q$-space  $U_0$ is spanned by the $e_{i,j}$'s with $i < j$. 
  Observing that for every pair $i,j$  we have $e_{i,j}^b = e_{\pi(i), \pi(j)}$ and that the pair $\{e_{i,j}, e_{j,i}\}^b = \{e_{\pi(i_,\pi(j)}, e_{\pi(j),\pi(i)}\}$
  contains exactly one element in $V^b \cap U$, we conclude that the $F_q$-spaces $V^b \cap U_0$ and $V \cap U_0$ have the same dimension.
  Therefore,  $|C \cap U| = |M \cap U| = |M^b \cap U| = |C^b \cap U|$

  We finish by noticing that   $|C_{U(n,q)}(d)| = \prod_i q^{\binom{a_i}{2}}$ is independent on the choice of the elements $d_1, \ldots, d_k$ and that every
  element in $D_{\delta}$ is conjugate by a permutation matrix to a 'standard-form' element $d$ as above.
\end{proof}

Before proving Theorem~B, we recall some elementary facts.

For any fixed prime number $p$ and an element (of finite order) $g$ of a group $G$, we can write in a unique way  $g = x y$, where $x$ and $y$ commute,
$x$ is a $p$-element and $y$ has order coprime to $p$. We call $x$ the $p$-part of $g$ and $y$ the $p'$-part of $g$.

Also, if $g, h \in G$ are elements of coprime order and they commute, then $C_G(gh) = C_G(g) \cap C_G(h)$.

\begin{theorem}\label{formula}
Let $T = T(n, q)$, $U = U(n,q)$ and $D = D(n,q)$. Then, 
$$|T^p| = \sum_{\delta\  \vdash n} \mid D_{\delta}\mid [U: C_{\delta}] \mid C_{\delta}^p\mid $$
where the sum runs over all partitions $\delta$ of $n$ with length $\leq q-1$,  $D_{\delta} = \{d\in D \mid \tau(d)=\delta\}$ and $C_{\delta} = C_{U}(d_{\delta})$ for some $d_{\delta} \in D_{\delta}$. 
\end{theorem}
\begin{proof}
  We first prove that
  \begin{equation}\label{UT}
T^p = \bigcup_{d\in D, v \in C_U(d)^p} (dv)^U 
\end{equation}
where $(dv)^U = \{ (dv)^y \mid y \in U\}$ is an orbit under the action by conjugation of $U$; we call it a $U$-class.

To prove~(\ref{UT}), let us consider an element $x = dv$ on the right hand side where $d \in D$ and $v= u^p$ for some $u \in C_{U}(d)$.
Since $(p, |D|) = 1$ we can write $d = d_0^p$ for some suitable $d_0 \in D$; note that  $u \in C_U(d_0) = C_U(d)$. Hence, $x =  d_0^pu^p = (d_0u)^p \in T^p$.
Since $T^p$ is $U$-invariant,  this proves that $T^p$ contains the union on the right side of~(\ref{UT}).
Conversely, consider $x = t^p$ with $t \in T$. Now  write $t =  d_0u$ where $u$ and $d_0$ are the $p$-part and $p'$-part of $t$, respectively.
So, in particular,  $u$ and $d_0$ commute. Let $y \in U$ such that $d_0^y \in D$ (such an element certainly exists, as $D$ is a $p$-complement of $T$ and the $p$-complements of $T$ are a single orbit under the action of $U$).
Write $d_1 = d_0^y$ and $v = u^y$. Then $x^y = (t^p)^y = (t^y)^p = ((d_0u)^y)^p = (d_1v)^p =  d_1^pv^p$. Now $d = d_1^p \in D$ and $C_U(d) = C_U(d_1)$. This proves the other inclusion. 

Next, we observe that for elements $u,v \in U$ and $c,d \in D$ which satisfy $[u, c] = 1 = [v,d]$ and $(cu)^U = (dv)^U$ then $c = d$. Let $y \in U$ such that $x = dv = c^yu^y$.
Note that $v$ and $d$ are the $p$-part and $p'$-part of $x$, respectively, and that the same is true for $u^y$ and $c^y$.
By uniqueness of $p$ and $p'$-parts, we hence get $v = u^y$ and $d = c^y$. In particular, $c^{-1}d = c^{-1}c^y = [c,y] \in D \cap U = 1$, so $c = d$.  

We also  have that if  $cu \in (dv)^U$,  for $c,d \in D$, $u \in U$ and $v \in C_U(d)^p$, then  $cu = du  = dv^y$ for some $y \in C_G(d)$.
Therefore, the family of $U$-classes in $T^p$ is in bijection with the set of pairs $(d, v_{d, i})$, where $d\in D$ and $v_{d, 1}, \ldots, v_{d, m_d}$ is a set of representatives
of the $C_U(d)$-classes in $C_U(d)^p$. 
For a fixed $d \in D$, write $C = C_U(d)$ and let $v \in C^p$. Observe that 
$C_U(dv) = C_U(d) \cap C_U(v) = C_C(v)$, because $d$ and $v$ are commuting elements of coprime order,  so 
$|(dv)^U| =  |U:C||v^{C}|$.
Hence,  we have  
$$\mid \bigcup_{ v \in C^p} (dv)^U  \mid = [U: C] \sum_{i = 1}^{m_d} \mid v_{d,i}^{C}\mid = [U:C]\mid C^p \mid .$$

By  Lemma~\ref{cent-sizes} we conclude that 

$$\mid T^p \mid = 
\sum_{d \in D} [U:C_U(d)]|C_U(d)^p| = 
\sum_{\delta\ \vdash n, l(\delta) < q} \mid D_{\delta}\mid [U: C_{\delta}] \mid C_{\delta}^p\mid $$
where $C_{\delta} = C_U(d_{\delta})$ for any fixed  $d_{\delta} \in D_{\delta}$.
\end{proof}

We will now prove Theorem~B.

\begin{proof}[Proof of the Theorem~B]
Proof the the theorem is obtained by simply substituting the values in the formula obtained the Theorem~\ref{formula} above. The value of $|D_{\delta}|$ is computed in the Lemma~\ref{type-count}. The value of $[U:C_{\delta}]$ is obtained from Lemma~\ref{cent-sizes}. Now to obtain the last term $|C_{\delta}^p|$ we use the fact that $C_{\delta}\cong \prod_{i=1}^k U(a_i,q)$. This completes the proof.
\end{proof}

Next, we apply the formula obtained in Theorem~B to compute some examples.

\begin{example}
Let $n = 3$, $q = 5 = p$. Let $T = T(3,5)$ and we want to compute $|T^5|$. 
In this case the partitions $\delta$  such that  $1 \leq l(\delta) \leq \min(n, q-1) = 3$)  are $(3)$, $(2,1)$ and $(1,1,1)$.
Now, $|\Delta_{(3)}|=4$, $|\Delta_{(2,1)}| = 3 \cdot (4\cdot 3) = 2^23^2$ and $|\Delta_{(1,1,1)}| = 4 \cdot 3 \cdot 2 = 2^33$.
Further $|d^T| = [U:C_U(d)]$ is, according to type, as follows:   $1$ for $(5)$, $5^2$ for $(2,1)$ and $5^3$ for $(1,1,1)$.
Hence,
$$|T^5| = 4\cdot 1 +  2^23^2\cdot 5^2 + 2^33\cdot 5^3 = 3904.$$
\end{example}

\begin{example} Let $n = 6$, $q = p = 3$,  $T = T(6,3)$,   $D =D(6,3)$ and $U = U(6,3)$.  

  The partitions $\delta$ of $6$ of length at most two are
  $(6), (5,1), (4, 2), (3,3)$ and for $d_{\delta} \in D_{\delta}$ we have the following

  \begin{center}
    \begin{tabular}{|c|c|c|c|c|}
      \hline
      & $(6)$  & $(5, 1)$ & $(4,2)$  & $(3,3)$ \\
      \hline
   $|D_{\delta}|$   & $2$ & $12$ & $30$  & $20$ \\
            \hline
      $[U: C_U(d_{\delta})]$  & $1$ & $3^5$ &  $3^8$ & $3^9$ \\
            \hline
            
        $|C_U(d_{\delta})^3|$ & $585$ & $3^3$ &  $3$ & $1$ \\
            \hline
            
    \end{tabular}
    \end{center}
  where we have used the fact that $C_U(d_{(6)}) \cong U(6,3)$ and that
  $|U(6, 3)^3| = 585$ (by direct computation).

  Hence, we get
$$|T^3| = 2\cdot 585 + 12 \cdot 3^5 \cdot 3^3 + 30 \cdot 3^8 \cdot 3 + 20 \cdot 3^9 = 1064052 .$$
\end{example}

We finish by proving Corollary~C.

\begin{proof}[Proof of Corollary~C]
  We will consider just  the partitions (of length $2$) $(n-i, i)$ for $1\leq i \leq \lfloor{n/2}\rfloor$.
  Hence, by Theorem~A and Theorem~B we have 
\begin{eqnarray*}
|T^p| & \geq& \sum_{i=1}^{\lfloor {n/2}\rfloor} \frac{(q-1)(q-2)n!}{2(n-i)!i!}.|U(i,q)^p||U(n-i,q)^p|\cdot q^{{n\choose 2} - {n-i \choose 2} - {i\choose 2}}\\
&\geq & \frac{1}{3^2}\sum_{i=1}^{\lfloor {n/2}\rfloor} \frac{(q-1)(q-2)}{2}{n\choose i}\cdot q^{{n\choose 2} - {n-i \choose 2} - {i\choose 2}+{i-p+1 \choose 2} + {n-i-p+1 \choose 2}}\\
&=& \frac{1}{3^2}\sum_{i=1}^{\lfloor {n/2}\rfloor} \frac{(q-1)(q-2)}{2}{n\choose i}\cdot q^{{n\choose 2} - (p-1)(n-p)} \\
&=& \frac{2^{n-1}}{3^2} (q-1)(q-2). q^{{n\choose 2} - (p-1)(n-p)}
\end{eqnarray*}
Hence,
$$\frac{|T^p|}{|T|}\geq \frac{2^{n-2} (q-1)(q-2)}{9(q-1)^n q^{ (p-1)(n-p)}} \geq  \frac{2^{n-2}}{9(q-1)^{n-2} q^{ (p-1)(n-p)}}.$$
\end{proof}



\begin{thebibliography}{99}
\normalsize
\bibitem[Bi]{bi} Bier, Agnieszka, \emph{``The width of verbal subgroups in the group of unitriangular matrices over a field''}, Internat. J. Algebra Comput.  22 (2012), no. 3, 1250019, 20 pp.
\bibitem[MAGMA]{magma} W. Bosma, J. Cannon and C. Playoust, \emph{``The Magma algebra system. I. The user language''}, Journal of Symbolic Computation 24 (1997), 235-265.
\bibitem[GKP]{gkp} Gordeev, N. L.; Kunyavski\u{i}, B. È.; Plotkin, E. B., \emph{``Word maps and word maps with constants of simple algebraic groups''}, Dokl. Akad. Nauk 471 (2016), no. 2, 136-138; translation in Dokl. Math. 94 (2016), no. 3, 632-634. 
\bibitem[GKP1]{gkp1} Gordeev, Nikolai; Kunyavski\u{i}, Boris; Plotkin, Eugene, \emph{``Word maps, word maps with constants and representation varieties of one-relator groups''}, J. Algebra 500 (2018), 390-424.
\bibitem[GKP2]{gkp2} Gordeev, Nikolai; Kunyavski\u{i}, Boris; Plotkin, Eugene, \emph{``Word maps on perfect algebraic groups''}, Internat. J. Algebra Comput. 28 (2018), no. 8, 1487-1515. 
\bibitem[GKP3]{gkp3} Gordeev, Nikolai; Kunyavski\u{i}, Boris; Plotkin, Eugene, \emph{``Geometry of word equations in simple algebraic groups over special fields''},  Uspekhi Mat. Nauk 73 (2018), no. 5(443), 3-52; translation in Russian Math. Surveys 73 (2018), no. 5, 753-796. 
\bibitem[GKSV]{gksv} Alexey Galt, Amit Kulshrestha, Anupam Singh, Evgeny Vdovin, \emph{``On Shalev's conjecture for type $A_n$ and ${}^{2}A_n$''}, accepted for publication in the Journal of Group Theory 2019.
\bibitem[Hu]{Hu} Huppert, B., \emph{``Endliche Gruppen. I''},  (German) Die Grundlehren der Mathematischen Wissenschaften, Band 134 Springer-Verlag, Berlin-New York 1967 xii+793 pp. 
\bibitem[Sh]{sh} Shalev, Aner, \emph{``Some results and problems in the theory of word maps''}, Erd\"{o}s centennial, 611-649, Bolyai Soc. Math. Stud., 25, J\'{a}nos Bolyai Math. Soc., Budapest, 2013. 
\bibitem[Sf]{sf} Soffer, Andrew, \emph{``Upper bounds on the number of conjugacy classes in unitriangular groups''}, J. Group Theory 19 (2016), no. 6, 1063-1095. 
\bibitem[So]{so} Sosnovskiy, Yury Vasil'evich, \emph{``On the width of verbal subgroups of the groups of triangular matrices over a field of arbitrary characteristic''}, Internat. J. Algebra Comput. 26 (2016), no. 2, 217-222. 
\bibitem[VA1]{va1} Vera L\'{o}pez, Antonio; Arregi, Jesus Maria, \emph{``Conjugacy classes in Sylow $p$-subgroups of $GL(n,q)$''},  J. Algebra 152 (1992), no. 1, 1-19. 
\bibitem[VA2]{va2} Vera-L\'{o}pez, A.; Arregi, J. M., \emph{``Conjugacy classes in Sylow p-subgroups of GL(n,q). II''}, Proc. Roy. Soc. Edinburgh Sect. A 119 (1991), no. 3-4, 343-346. 
\end{thebibliography}
\end{document}